\documentclass{amsart}

\usepackage{multicol,fullpage}
\usepackage{amsmath,amsfonts,latexsym,amssymb,enumerate}

\newtheorem{theorem}{Theorem}[section]

\newtheorem{proposition}[theorem]{Proposition}
\newtheorem{corollary}[theorem]{Corollary}

\theoremstyle{definition}

\theoremstyle{remark}

\newcommand{\bbZ}{\mathbb{Z}}
\newcommand{\Bkmq}{B^{({\mathbf{k}})}_m(q)}
\newcommand{\bkmn}{b^{({\mathbf{k}})}_m(n)}
\newcommand{\Ckmq}{C^{({\mathbf{k}})}_m(q)}
\newcommand{\ckmn}{c^{({\mathbf{k}})}_m(n)}
\newcommand{\ckm}{c^{({\mathbf{k}})}_m}
\newcommand{\bfk}{{\mathbf{k}}}

\newcommand{\mm}{\;({\mathrm{mod}} \; m)}

\newcommand{\ld}{\ldots}

\newcommand{\vep}{\varepsilon}

\begin{document}

\title{Characterizing the number of coloured $m-$ary partitions modulo $m$, with and without gaps}

\date{\today}

\author{I. P. Goulden}
\address{Dept. of Combinatorics and Optimization, University of Waterloo, Canada}
\curraddr{}
\email{ipgoulde@uwaterloo.ca, pavel.shuldiner@gmail.com}

\author{Pavel Shuldiner}
\thanks{The work of IPG was supported by an NSERC Discovery Grant.}
 


\begin{abstract}
In a pair of recent papers, Andrews, Fraenkel and Sellers provide a complete characterization for the number of~$m$-ary partitions modulo~$m$, with and without gaps. In this paper we extend these results to the case of coloured~$m$-ary partitions, with and without gaps. Our method of proof is different, giving explicit expansions for the generating functions modulo~$m$.
\end{abstract}
 
\maketitle 

\section{Introduction}

An $m$-ary partition is an integer partition in which each part is a nonnegative integer power of a fixed integer~$m\ge 2$. An $m$-ary partition {\em without gaps} is an $m$-ary partition in which $m^j$ must occur as a part whenever $m^{j+1}$ occurs as a part, for every nonnegative integer $j$.

Recently, Andrews, Fraenkl and Sellers~\cite{afs15} found an explicit expression that characterizes the number of $m$-ary partitions of a nonnegative integer $n$ modulo $m$; remarkably, this expression depended only on the coefficients in the base $m$ representation of $n$. Subsequently Andrews, Fraenkel and Sellers~\cite{afs16} followed this up with a similar result for the number of $m$-ary partitions without gaps, of a nonnegative integer $n$ modulo $m$; again, they were able to obtain a (more complicated) explicit expression, and again this expression depended only on the coefficients in the base $m$ representation of $n$. See also Edgar~\cite{e16} and Ekhad and Zeilberger~\cite{ez15} for more on these results.

The study of congruences for integer partition numbers has a long history, starting with the work of Ramanujan (see, e.g.,~\cite{r19}). For the special case of~$m$-ary partitions, a number of authors have studied congruence properties, including Churchhouse~\cite{c69} for~$m=2$, R\o dseth~\cite{r70} for~$m$ a prime, and Andrews~\cite{a71} for arbitrary positive integers $m\ge 2$. The numbers of~$m$-ary partitions without gaps had been previously considered by Bessenrodt, Olsson and Sellers~\cite{bos13} for~$m=2$.

In this note, we consider~$m$-ary partitions, with and without gaps, in which the parts are {\em coloured}. To specify the number of colours for parts of each size, we let $\bfk = (k_0,k_1,\ld )$ for positive integers $k_0,k_1,\ld $, and say that an~$m$-ary partition is $\bfk$-{\em coloured} when there are $k_j$ colours for the part $m^j$, for $j\ge 0$. This means that there are~$k_j$ different kinds of parts of the same size $m^j$. Let $\bkmn$ denote the number of $\bfk$-coloured~$m$-ary partitions of $n$, and let $\ckmn$ denote the number of $\bfk$-coloured~$m$-ary partitions of~$n$ without gaps. For the latter, some part $m^j$ of any colour must occur as a part whenever some part $m^{j+1}$ of any colour (not necessarily the same colour) occurs as a part, for every nonnegative integer $j$.

 We extend the results of Andrews, Fraenkel and Sellers in~\cite{afs15} and~\cite{afs16} to the case of $\bfk$-coloured~$m$-ary partitions, where~$m$ is relatively prime to $(k_0-1)!$ and to $k_j!$ for $j\ge 1$. Our method of proof is different, giving explicit expansions for the generating functions modulo $m$. These expansions depend on the following simple result.

\begin{proposition}\label{binequiv}
For positive integers $m,a$ with $m$ relatively prime to $(a-1)!$, we have
\[ \left(1-q \right)^{-a} \equiv  \left(1-q^m \right)^{-1} \sum_{\ell = 0}^{m-1} {a-1+\ell \choose a-1} q^{\ell}  \mm .  \]
\end{proposition}

\begin{proof}
From the binomial theorem we have
\[ \left(1-q \right)^{-a} = \sum_{\ell = 0}^{\infty} {a-1+\ell \choose a-1} q^{\ell}. \]
Now using the falling factorial notation $(a-1+\ell)_{a-1} = (a-1+\ell)(a-2+\ell)\cdots (1+\ell)$ we have
\[ {a-1+\ell \choose a-1} = \left( (a-1)!\right)^{-1}(a-1+\ell)_{a-1}. \]
But
\[ (a-1+\ell +m)_{a-1} \equiv  (a-1+\ell)_{a-1} \mm, \]
for any integer $\ell$, and $\left( (a-1)!\right)^{-1} $ exists in $\bbZ_m$ since $m$ is relatively prime to $(a-1)!$, which gives
\begin{equation}\label{bincoeffequiv}
{a-1+\ell + m\choose a-1} \equiv {a-1+\ell \choose a-1} \mm,
\end{equation}
and the result follows.
\end{proof}

\section{Coloured $m$-ary partitions}

In this section we consider the following generating function for the numbers~$\bkmn$ of~$\bfk$-coloured~$m$-ary partitions:
\[ \Bkmq = \sum_{n = 0}^{\infty} \bkmn q^n = \prod_{j=0}^{\infty} \left( 1 - q^{m^j}\right)^{-k_j}.  \]
The following result gives an explicit expansion for~$\Bkmq$ modulo~$m$.

\begin{theorem}\label{Bexpn}
If~$m$ is relatively prime to $(k_0-1)!$ and to $k_j!$ for $j\ge 1$, then we have
\[ \Bkmq \equiv \left( \sum_{\ell_0=0}^{m-1}{k_0-1+\ell_0 \choose k_0-1} q^{\ell_0} \right) \prod_{j=1}^{\infty} \left( \sum_{\ell_j=0}^{m-1}{k_j+\ell_j \choose k_j} q^{\ell_j m^j}\right) \mm.  \]
\end{theorem}

\begin{proof}
Consider the finite product
\[ P_i= \prod_{j=0}^{i} \left( 1 - q^{m^j}\right)^{-k_j},\qquad i\ge 0.  \]
We prove that
\begin{equation}\label{finnogap}
 P_i \equiv \left( \sum_{\ell_0=0}^{m-1}{k_0-1+\ell_0 \choose k_0-1} q^{\ell_0} \right) \left( 1-q^{m^{i+1}} \right)^{-1} \prod_{j=1}^{i} \left( \sum_{\ell_j=0}^{m-1}{k_j+\ell_j \choose k_j} q^{\ell_j m^j}\right) \mm,
 \end{equation}
by induction on $i$. As a base case, the result for $i=0$ follows immediately from Proposition~\ref{binequiv} with $a=k_0$. Now assume that~(\ref{finnogap}) holds for some choice of $i\ge 0$, and we obtain
\begin{align*}
P_{i+1}&=  \prod_{j=0}^{i+1} \left( 1 - q^{m^j}\right)^{-k_j} = \left(1-q^{m^{i+1}} \right)^{-k_{i+1}} P_i \\
&\equiv  \left( \sum_{\ell_0=0}^{m-1}{k_0-1+\ell_0 \choose k_0-1} q^{\ell_0} \right) \left( 1-q^{m^{i+1}} \right)^{-k_{i+1}-1} \prod_{j=1}^{i} \left( \sum_{\ell_j=0}^{m-1}{k_j+\ell_j \choose k_j} q^{\ell_j m^j}\right)  \mm \\
&\equiv \left( \sum_{\ell_0=0}^{m-1}{k_0-1+\ell_0 \choose k_0-1} q^{\ell_0} \right) \left( 1-q^{m^{i+2}} \right)^{-1} \prod_{j=1}^{i+1} \left( \sum_{\ell_j=0}^{m-1}{k_j+\ell_j \choose k_j} q^{\ell_j m^j}\right) \mm ,
\end{align*} 
where the second last equivalence follows from the induction hypothesis, and the last equivalence follows from Proposition~\ref{binequiv} with $a=k_{i+1}+1$, $q=q^{m^{i+1}}$.

This completes the proof of~(\ref{finnogap}) by induction on $i$, and the result follows immediately since
\[ \Bkmq = \lim_{i \rightarrow \infty} P_i . \]
\end{proof}

Now we give the explicit expression for the coefficients modulo $m$ that follows from the above expansion of the generating function $\Bkmq$.

\begin{corollary}\label{bexpn}
For $n\ge 0$, suppose that the base $m$ representation of $n$ is given by
\[ n=d_0+d_1m +\ld + d_t m^t, \qquad 0 \le t. \]
If~$m$ is relatively prime to $(k_0-1)!$ and to $k_j!$ for $j\ge 1$, then we have
\[  \bkmn \equiv  {k_0-1+d_0 \choose k_0-1} \prod_{j=1}^{t}  {k_j+d_j \choose k_j} \mm.  \]
\end{corollary}

\begin{proof}
In the expansion of the series $\Bkmq$ given in Theorem~\ref{Bexpn}, the monomial $q^n$ arises uniquely with the specializations $\ell_j=d_j$, $j = 0, \ld ,t$ and $\ell_j = 0$, $j\ge t$. But with these specializations, we have ${k_j+\ell_j \choose k_j}={k_j \choose k_j}=1$, and the result follows immediately. 
\end{proof}

Specializing the expression given in~Corollary~\ref{bexpn} to the case $k_j=1$ for $j\ge 0$ provides an alternative proof to Andrews, Fraenkel and Sellers' characterization of $m-$ary partitions modulo $m$, which was given as Theorem~1 of~\cite{afs15}.

\section{Coloured $m$-ary partitions without gaps}

In this section we consider the following generating function for the numbers~$\ckmn$ of~$\bfk$-coloured~$m$-ary partitions without gaps:
\[ \Ckmq = 1 + \sum_{n=0}^{\infty} \ckmn q^n = 1 + \sum_{i=0}^{\infty} \prod_{j=0}^i \left( \left(1-q^{m^j} \right)^{-k_j}   -1 \right)  . \]
The following result gives an explicit expansion for~$\Ckmq$ modulo~$m$.

\begin{theorem}\label{Cexpn}
If~$m$ is relatively prime to $(k_0-1)!$ and to $k_j!$ for $j\ge 1$, then we have
\[ \Ckmq \equiv 1 + \left( \sum_{\ell_0=1}^{m} {k_0-1+\ell_0 \choose k_0-1} q^{\ell_0} \right) \sum_{i=0}^{\infty} \left( 1-q^{m^{i+1}} \right)^{-1} \prod_{j=1}^i \left( \sum_{\ell_j=0}^{m-1}\bigg\{  {k_j+\ell_j \choose k_j }   -1 \bigg\} q^{\ell_j m^j} \right) \mm .   \]
\end{theorem}

\begin{proof}
Consider the finite product
\[ R_i= \prod_{j=0}^{i} \left( \left( 1 - q^{m^j}\right)^{-k_j} - 1 \right), \qquad i\ge 0.  \]
We prove that
\begin{equation}\label{fingap}
 R_i \equiv \left( \sum_{\ell_0=1}^{m}{k_0-1+\ell_0 \choose k_0-1} q^{\ell_0} \right) \left( 1-q^{m^{i+1}} \right)^{-1}\prod_{j=1}^{i} \left( \sum_{\ell_j=0}^{m-1} \bigg\{ {k_j+\ell_j \choose k_j} - 1 \bigg\} q^{\ell_j m^j}\right) \mm,
 \end{equation}
by induction on $i$. As a base case, the result for $i=0$ follows immediately from Proposition~\ref{binequiv} with $a=k_0$. Now assume that~(\ref{fingap}) holds for some choice of $i\ge 0$, and we obtain
\begin{align*}
R_{i+1}  &=  \prod_{j=0}^{i+1} \left( \left( 1 - q^{m^j}\right)^{-k_j} - 1 \right) = \left( \left(1-q^{m^{i+1}} \right)^{-k_{i+1}} - 1 \right) R_i \\
&\equiv  \left( \sum_{\ell_0=1}^{m}{k_0-1+\ell_0 \choose k_0-1} q^{\ell_0} \right) \bigg\{ \left(1-q^{m^{i+1}} \right)^{-k_{i+1}-1}-\left( 1-q^{m^{i+1}} \right)^{-1} \bigg\} \\
& \qquad\qquad\qquad\qquad \qquad\qquad\qquad\qquad \qquad\qquad\qquad\qquad  \times \; \prod_{j=1}^{i} \left( \sum_{\ell_j=0}^{m-1} \bigg\{ {k_j+\ell_j \choose k_j} - 1 \bigg\} q^{\ell_j m^j}\right)  \mm \\
&\equiv \left( \sum_{\ell_0=1}^{m}{k_0-1+\ell_0 \choose k_0-1} q^{\ell_0} \right) \left( 1-q^{m^{i+2}} \right)^{-1} \prod_{j=1}^{i+1} \left( \sum_{\ell_j=0}^{m-1} \bigg\{ {k_j+\ell_j \choose k_j} - 1 \bigg\} q^{\ell_j m^j}\right) \mm ,
\end{align*} 
where the second last equivalence follows from the induction hypothesis, and the last equivalence follows from Proposition~\ref{binequiv} with $a=k_{i+1}+1$, $q=q^{m^{i+1}}$ and $a=1$, $q=q^{m^{i+1}}$.

This completes the proof of~(\ref{fingap}) by induction on $i$, and the result follows immediately since
\[ \Ckmq =  1 + \sum_{i=0}^{\infty} R_i . \]

\end{proof}

\begin{corollary}\label{cexpn}
For $n\ge 1$, suppose that $n$ is divisible by $m$, with  base $m$ representation given by
\[ n = d_s m^s +\ld + d_t m^t, \qquad\quad 1 \le s \le t, \]
where $1 \le d_s \le m-1$, and $0\le d_{s+1},\ld ,d_t \le m-1$. If~$m$ is relatively prime to $(k_0-1)!$ and to $k_j!$ for $j\ge 1$, then for $0\le d_0 \le m-1$ we have 
\[  \ckm(n-d_0) \equiv  {k_0-1- d_0 \choose k_0-1}\left( \vep_s + (-1)^{s-1} \bigg\{  {k_s+d_s - 1 \choose k_s} -1 \bigg\} \sum_{i=s}^t \prod_{j= s+1}^i  \bigg\{ {k_j+d_j \choose k_j} -1 \bigg\} \right) \mm,  \]
where $\vep_s=0$ if $s$ is even, and $\vep_s = 1$ if $s$ is odd.
\end{corollary}

\begin{proof}
First note that we have
\[ n - d_0 = m - d_0 +(m-1)m^1 + \ld + (m-1)m^{s-1} + (d_s-1)m^s + d_{s+1} m^{s+1} + \ld + d_t m^t. \]
Now consider the following specializations: $\ell_0 = m-d_0$,  $\ell_j=m-1$, $j=1,\ld ,s-1$, $\ell_s = d_s-1$, $\ell_j=d_j$, $j=s+1,\ld ,t$, and $\ell_j=0$, $j > t$. Then, in the expansion of the series $\Ckmq$ given in Theorem~\ref{Cexpn}, the monomial $q^n$ arises once for each $i\ge 0$, in particular with the above specializations truncated to $\ell_0,\ld ,\ell_i$. But with these specializations we have

\vspace{.1in}

\begin{itemize}

\item
for $j=0$:
\[ {k_j - 1+ \ell_j \choose k_j - 1} - 1 = {k_0-1+m-d_0 \choose k_0-1}={k_0-1-d_0 \choose k_0-1}, \qquad \mathrm{from}\;\; ~(\ref{bincoeffequiv}), \]

\item
for $j= 1,\ld ,s-1$:
\[{k_j + \ell_j \choose k_j} - 1 = {k_j - 1 \choose k_j} - 1 = 0 - 1 = -1, \]
and
\[ \sum_{i=0}^{s-1} \prod_{j=1}^i \bigg\{  {k_j+\ell_j \choose k_j }   -1 \bigg\} = \sum_{i=0}^{s-1} (-1)^i = \vep_s, \]

\item 
for $j=s$:
\[ {k_j + \ell_j \choose k_j} - 1 = {k_s + d_s - 1 \choose k_s} - 1, \]

\item
for $j=s+1,\ld ,t$:
\[ {k_j + \ell_j \choose k_j} - 1 = {k_j + d_j \choose k_j} - 1, \]

\item
for $j> t$:
\[ {k_j+\ell_j \choose k_j} - 1 = {k_j \choose k_j} - 1 = 1 - 1 = 0. \]

\end{itemize}

\vspace{.1in}

The result follows straightforwardly from Theorem~\ref{Cexpn}. 
\end{proof}

Specializing the expression given in~Corollary~\ref{cexpn} to the case $k_j=1$ for $j\ge 0$ provides an alternative proof to Andrews, Fraenkel and Sellers' characterization of $m-$ary partitions modulo $m$ without gaps, which was given as Theorem~2.1 of~\cite{afs16}.

\bibliographystyle{amsplain}

\end{document}